\documentclass[preprint,12pt]{elsarticle}

\journal{Journal of Number Theory}

\usepackage{amsmath, amssymb, amsfonts, amsthm}
\usepackage[cp1251]{inputenc}  
\usepackage[T2A]{fontenc}     
\usepackage{indentfirst}
\usepackage{color}
\usepackage{enumerate}

\newtheorem{theorem}{Theorem}[section]

\newtheorem{proposition}[theorem]{Proposition}

\theoremstyle {definition}
\newtheorem{example}[theorem]{Example}
\newtheorem{remark}[theorem]{Remark}

\numberwithin{equation}{section}

\newcommand{\Z}{\mathbb Z}
\newcommand{\Q}{\mathbb Q}
\newcommand{\N}{\mathbb N}

\renewcommand{\:}{\colon}
\renewcommand{\>}{\rightarrow}

\title{Groups of automorphisms of $p$-adic integers and the problem of the existence of fully homomorphic ciphers}

\author{Ekaterina Yurova Axelsson} 
\ead{Ekaterina.Yurova@lnu.se}
\address{International Center for Mathematical Modeling in Physics, Engineering, 
Economics, and Cognitive Science,  
Linnaeus University, Sweden}

\author{Andrei Khrennikov}
\ead{Andrei.Khrennikov@lnu.se}
\address{International Center for Mathematical Modeling in Physics, Engineering, 
Economics, and Cognitive Science,  
Linnaeus University, Sweden}

\begin{document}

\begin{abstract}

In this paper, we study groups of automorphisms  of algebraic systems over a set of $p$-adic integers with different sets of arithmetic and coordinate-wise logical operations and congruence relations modulo $p^k,$ $k\ge 1.$ The main result of this paper is the description of groups of automorphisms of $p$-adic integers with one or two arithmetic or coordinate-wise logical operations on $p$-adic integers. To describe groups of automorphisms, we use the apparatus of the $p$-adic analysis and $p$-adic dynamical systems.

The motive for the study of  groups of automorphism of algebraic systems over $p$-adic integers is the question of the existence of a fully homomorphic encryption in a given family of ciphers. The relationship between these problems is based on the possibility of constructing a "continuous" $p$-adic model for some families of ciphers (in this context, these ciphers can be considered as "discrete" systems). As a consequence, we can apply the "continuous" methods of $p$-adic analysis to solve the "discrete" problem of the existence of fully homomorphic ciphers.

\end{abstract}



\maketitle

\section{Introduction}
In this paper, we study groups of automorphisms of $p$-adic integers $\Z_p.$ We consider the set $\Z_p$ as an algebraic system with a given set of binary operations and relations (or predicates). We recall that the algebraic system is a triple $\mathcal{A}=\langle A, \Omega_{\mathcal A}, P_{\mathcal A}\rangle$, where $A$ is a set (i.e., a carrier of system $\mathcal{A}$), $\Omega_{\mathcal A}$ is a set of operations (in our case binary) on $A$  (i.e., an operator domain), and $P_{\mathcal A}$ is a set of relations (in our case binary) on $A$ (i.e., a predicate domain), see, for example, \cite {Coh} and \cite {Maltcev}. A predicate on $A$ (in our case binary) is a mapping $\pi: A\times A\to \{True, False\}.$ We denote a predicate as $\pi(x,y)$ instead of $\pi(x,y)=True$. In fact, the predicate is the characteristic function of some subset of $A\times A,$ i.e. relations on $A.$ Therefore, the concepts of relation and predicate are treated as synonyms.

An automorphism of an algebraic system $\mathcal{A}$ is a bijective mapping $\phi: A\to A$ such that $\phi(x\star y)=\phi(x)\star \phi(y)$, $x, y\in A$ for any operation $\star\in \Omega_{\mathcal A}.$ Moreover, if $\pi(x,y)$, then $\pi(\phi(x),\phi(y))$ for any predicate $\pi\in P_{\mathcal A}$, $x, y\in A$ (in other words, $\phi$ preserves all the operations and predicates (or relations)).

For $p$-adic integers, we consider the algebraic system of the following form $\mathcal{A}=\langle A, \Omega_{\mathcal A}, P_{\mathcal A}\rangle,$ where $A=\Z_p,$ predicate domain $P_{\mathcal A}$ is determined by the congruence relations  modulo $p^k,$ $k\ge 1,$ and operator domain $\Omega_{\mathcal A}$ consists of one or two operations from the set $O_{\Z_p}=\{+, \cdot, \mathrm{XOR}, \mathrm{AND}\}.$  Operations "$+$" and "$\cdot$" are arithmetic operations on $\Z_p.$  Coordinate-wise logical operations "$\mathrm{XOR}$" and "$\mathrm{AND}$" are also given on $\Z_p.$ Their meaning is to implement the logical operations of addition and multiplication on the set $\{0,\ldots, p-1\}$ for each coordinate of the canonical representation of a $p$-adic integer (for more details, see Section \ref {Automorph_def}). To denote the algebraic systems under consideration, we shall use the notation $\mathcal{A}_p(*)$ for one operation and $\mathcal{A}_p(*_1,*_2)$ for two operations, where $*,*_1,*_2\in O_{\Z_p}.$

The main results are presented in Section \ref{sec_automorph}. In Theorems \ref{t_gom_ariff} and \ref{t_gom_coord}, we give a description of groups of automorphisms of algebraic systems of $p$-adic integers $\mathcal{A}_p(*),$ where $*\in O_{\Z_p}.$ Here "$*$" is one of the arithmetic ("$+$" and "$\cdot$") or coordinate-wise logical ("$\mathrm{XOR}$" and "$\mathrm{AND}$") operations. 

These results were obtained on the basis of the apparatus developed in our previous works on $p$-adic (and, especially, measure-preserving) dynamical systems \cite{MeraJNT}, \cite{SOL},  see also  pioneering papers of V. Anashin \cite{Tfunc}-\cite{An_avt_1} and monograph \cite{ANKH}. See also works  \cite{V00}, \cite{V1} on the general theory of $p$-adic dynamical system and more generally interrelation between number theory and dynamical systems. In particular, in terms of $p$-adic dynamics, an automorphism $\mathcal{A}_p(*)$ is a 1-Lipschitz measure-preserving function  $\Z_p\to \Z_p,$ that is a homomorphism with respect to a given operation "*". Here the condition "1-Lipschitz" corresponds to the preservation of the predicates $\mathcal{P}$ that define the congruence relations modulo $p^k,$ $k\geq 1$ and the condition "preserves the measure" corresponds to the bijectivity (reversibility) of the function whereby the automorphism is determined.

In Theorem \ref {Prop_aut_2}, we consider the case where any two operations from a set of arithmetic and coordinate-wise logical operations are defined on $\Z_p.$
It turned out that all groups of automorphisms  of algebraic system of $p$-adic integers $\mathcal{A}_p(*_1, *_2)$ for $*_1, *_2\in O_{\Z_p}$ are trivial. Due to the result of Theorem \ref {Prop_aut_2}, there arises the question of the existence of an algebraic system of $p$-adic integers $\mathcal{A}_p(g_1, g_2),$ where $g_1$ and $g_2$ are "new" operations for which the group of automorphisms differs from the trivial group. In Proposition \ref {new_operation}, we describe all the operations "$G$" (here $G: \Z_p\times \Z_p \to \Z_p$) on $\Z_p$ for which the groups of automorphisms  of the algebraic systems $\mathcal{A}_p(+, G)$ are not trivial (here operations "$G$" are given by a  convergent series on $\Z_p$).

We also consider the case where "new" operations are given as formulas in a basis of two arbitrary arithmetic and coordinate-wise logical operations over $\Z_p.$
In this case, the necessary condition for the non-triviality of the group of automorphisms  $\mathcal{A}_p(g_1, g_2)$ is that the set of formulas in the basis of the operations $g_1, g_2$ does not coincide with the set of formulas in the chosen basis of arithmetic or coordinate-wise logical operations over $\Z_p$ (see Proposition \ref{formuls}).

Our main reason to consider such groups of automorphisms of $p$-adic integers is the possibility of using the apparatus of $p$-adic analysis to introduce the transformations on $\Z_p,$ which can be used to construct fully homomorphic ciphers. Recall that a ciher is a family $f_r,$ $r\in R$ of bijective mappings of a set of open texts $X$ into a set of ciphered texts $Y$, where the parameter $r$ is a key. Note that in the general case $f_r$ only required property of injectivity, but usually, it is considered bijective transformation. We consider ciphers for which the sets $X$ and $Y$ coincide and consist of words of finite length in the alphabet $\mathrm B=\{0, 1,\ldots, p-1\}$ for prime number $p.$ In this case, if one operation (or two operations) on $X=Y$ is given and for any $r\in R$ the transformation $f_r$ is a homomorphism with respect to this operation (respectively, to these operations), then it is said that the cipher is homomorphic (respectively, fully homomorphic). The problem of constructing a fully homomorphic encryption is relevant for the secure cloud computing (for more details, see \ref {sec_model}).

It turns out that algebraic systems of $p$-adic integers $\mathcal{A}_p(*),$ $\mathcal{A}_p(*_1,*_2)$ for $*, *_1, *_2\in O_{\Z_p}$ are "continuous" $p$-adic models of the ciphers under consideration with operations that are discrete analogs of operations in $O_{\Z_p}.$ The description of ciphers for which there exist "continuous" $p$-adic models, as well as the rationale for the choice of such models, are presented in Section \ref {sec_model}. If there is a description of automorphism groups of $p$-adic integers $\mathcal{A}_p(*)$, $\mathcal{A}_p(*_1,*_2)$ in the framework of a "continuous" $p$-adic model, then, choosing the corresponding "discrete" analogues of these automorphisms, we can construct homomorphic (fully homomorphic) ciphers from the family of ciphers under consideration.

We recall some definitions related to the $p$-adic analysis and we introduce the necessary notations.


\subsection{$P$-adic numbers}
\label {p-adic}
For any prime number $p$ the $p$-adic norm $|\cdot|_p$ is defined on $\Q$ in the following way. For every nonzero integer $n$ let $ord_p(n)$ be the highest power of $p$ which divides $n$. Then we define $|n|_p=p^{-ord_p (n)}$, $|0|_p=0$ and $|\frac {n}{m}|_p=p^{-ord_p(n)+ord_p(m)}$.

The completion of $\Q$ with respect to the $p$-adic metric $\rho_p(x,y)=|x-y|_p$  is called the field of $p$-adic numbers $\Q_p$.
The metric $\rho_p$ satisfies the so-called strong triangle inequality $|x\pm y|_p\le \max{(|x|_p;|y|_p)}$. 
The set $\Z_p=\{x\in\Q_p\colon |x|_p\leq 1\}$ is called the set of $p$-adic integers. 

Hereinafter, we will consider only the $p$-adic integers.
Every $x\in \Z_p$ can be expanded in canonical form, namely, in the form of a series that converges for the $p$-adic norm:
$x=x_0+px_1+\ldots+p^kx_k+\ldots,\quad x_k\in\{0,1,\ldots,p-1\},$ $k\ge 0.$

Partial sums of this series, we denote as $[x]_k,$ i.e. $[x]_k=x_0+px_1+\ldots+p^{k-1}x_{k-1},$ $k\ge 1.$

If residues of the ring $\Z/p^k\Z$ are set as minimal non-negative integers, then for $x\in\Z_p$ we can consider notation $x \pmod {p^k}$ in the sense of
\begin{equation}
\label {_x}
    x \pmod {p^k}=[x]_k\;\; \text {or} \;\;\; x\equiv[x]_k\pmod {p^k}.
\end{equation}
Let $a\in \Z_p$ and $r$ be positive integers. The set $B_{p^{-r}}(a)=\{x\in\Z_p : |x-a|_p\le p^{-r}\}=a+p^r\Z_p$ is a ball of radius  $p^{-r}$ with a center $a$.

\subsection {$P$-adic functions}
\label {subsec_Lip}

In this paper, we consider functions $f: \Z_p\> \Z_p$, which satisfy the Lipschitz condition with a constant 1 (i.e., 1-Lipschitz functions).
Recall that 
$f: \Z_p\> \Z_p$ is a 1-Lipschitz function if $|f(x)-f(y)|_p\le |x-y|_p,$ for all $x,y\in\Z_p$.
This condition is equivalent to the following: $x\equiv y \pmod {p^k}$ implies $f(x)\equiv f(y) \pmod {p^{\;k}}$ for all $k\ge 1$.

For all $k\ge 1$ a 1-Lipschitz transformation $f: \Z_p \rightarrow \Z_p$ of the reduced mapping modulo $p^k$ is  
\begin{equation}
\label {f_mod_p_k}
    f_{k-1}: \Z/p^k\Z \rightarrow Z/p^k\Z,\;\;z\mapsto f(z)\pmod {p^k}.
\end{equation}
A mapping $f_{k-1}$ is well defined (i.e. the $f_{k-1}$ does not depend on the choice of representative in the ball $z + p^k\Z_p$). 
We use the notation $f_{k-1}\equiv f\pmod {p^k}$ taking into account (\ref{_x}).

\subsubsection{Van der Put series}
Continuous $p$-adic functions can be represented in the form of the van der Put series.
The van der Put series is defined in the following way. Let $f\:\Z_p\>\Z_p$ be a continuous function. Then there exists a unique sequence of $p$-adic coefficients
$B_0,B_1,B_2, \ldots$ such that 
\begin{equation}
\label{vdp} 
f(x)=\sum_{m=0}^{\infty} B_m \chi(m,x)     
\end{equation}
for all $x \in \Z_p.$
Here the characteristic function $\chi(m,x)$ is given by $\chi(m,x)=1$ if $|x-m|_p\le p^{-n}$ and $\chi(m,x)=0$ otherwise, where $n=0$ if $m=0$, and $n$ is uniquely defined by the inequality 
$p^{n-1}\leq m \leq p^n-1$ otherwise (see Schikhof's book \cite{Schikhof} for a detailed presentation of the theory of the van der Put series). 

The coefficients $B_m$ are related to the values of the function $f$ in the following way. Let $m=m_0+ \ldots +m_{n-2} p^{n-2}+m_{n-1} p^{n-1}$, 
$m_j\in \{0,\ldots ,p-1\}$, $j=0,1,\ldots, n-1$ and $m_{n-1}\neq 0,$ then $B_m=f(m)-f(m-m_{n-1} p^{n-1})$ if $m\ge p$ and $B_m=f(m)$ otherwise.

1-Lipschitz functions $f\: \Z_p \to \Z_p$ in terms of the van der Put series were described in \cite {Schikhof}. We follow Theorem 3.1 \cite {Tfunc} as a convenience for further study. In this theorem, the function $f$ presented via the van der Put series is 1-Lipschitz if and only if $|B_m|_p \leq p^{-\left\lfloor \log_p m \right\rfloor}$ for all $m \ge 0.$ Assuming $B_m=p^{\left\lfloor \log_p m\right\rfloor} b_m$, we find that the function $f$ is 1-Lipschitz if and only if it can be represented as
\begin{equation}
\label {1-Lip}
    f(x)=\sum_{m=0}^\infty p^{\left\lfloor \log_p m\right\rfloor} b_m \chi (m,x)
\end{equation}
for suitable $b_m \in \Z_p,$ $m \ge 0.$

\subsubsection {Coordinate representation of 1-Lipschitz functions}
In this section we describe a coordinate representation of $p$-adic functions, see, for example, \cite{YuRecent}.

Let functions $\delta_k (x),$ $k=0,1,2,\ldots$ be the $k$-th digit in a $p$-base expansion of the number $x\in \Z_p,$ i.e. $\delta_k \: \Z_p\rightarrow \left\{0, 1, \ldots, p-1\right\},$ $\delta_k(x)=x_k.$

Any map $f: \Z_p \to \Z_p$ can be represented in the form:
\begin{equation}
\label{coor1} 
f(x)=\delta_0 (f(x))+ p\delta_1 (f(x))+ \ldots+ p^k \delta_k (f(x))+\ldots.    
\end{equation}

According to Proposition 3.33 in \cite {ANKH}, $f$ is a 1-Lipschitz function if and only if
for every $k \ge 1$ the k-th coordinate function $\delta_k (f (x))$ does not
depend on $\delta_{k+s}(x)$ for all $s \ge 1$, i.e. $\delta_k(f(x+p^{k+1}\Z_p))=\delta_k(f(x))$ for all $x\in \{0, 1, \ldots, p^{k+1}-1\}$.

Taking into account notation (\ref{_x}) for $k\ge 0$, we consider the following functions of $p$-valued logic 
$$
\varphi_k\: \underbrace{\{0,\ldots,p-1\}\times\ldots\times\{0,\ldots,p-1\}}_{k+1} \> \{0,\ldots,p-1\},
$$ 
and $\varphi_k: (x_0, x_1, \ldots, x_k) \mapsto \delta_k(f(x)).$ 

Then any 1-Lipschitz function $f\: \Z_p \> \Z_p$ can be represented as

\begin{equation}
\label{coor2}   
 f(x)=f(x_0+\ldots+p^kx_k+\ldots)=\sum_{k=0}^{\infty}p^k\varphi_k(x_0,\ldots,x_k).
\end{equation}

The function $\varphi_k(x_0, \ldots, x_k)$ can be defined by its sub-functions obtained by fixing the first $k$ variables   $(x_0, \ldots, x_{k-1})$.
Sub-function of the function  $\varphi_k(x_0, \ldots, x_k)$ which is obtained by fixing the variables $x_0=a_0, \ldots, x_{k-1}=a_{k-1},$ $a_i\in\{0, \ldots, p-1\},$ is denoted by $\varphi_{k,a}$, where $a=a_0+pa_1+\ldots+p^{k-1}a_{k-1}$. 

Thus, we can rewrite the 1-Lipschitz function $f\: \Z_p \> \Z_p$ as
\begin{multline}
\label{coor3} 
    f(x)=f(x_0+px_1+\ldots+p^kx_k+\ldots)=\\
=\varphi_0(x_0)+\sum_{k=1}^{\infty} p^k\sum_{a=0}^{p^k-1}I_{a}([x]_k)\varphi_{k,a}(x_k),
\end{multline}
where $I_{a}([x]_k)=1$, if $[x]_k=a$ and $I_{a}([x]_k)=0$ otherwise.

We call the relation (\ref {coor3}) the sub-coordinate representation of a 1-Lipschitz function $f,$ see \cite {Subcoord_Hensel} and \cite {Subcoord_Hensel_1}. Functions $\varphi_{k,a}$, $\varphi_0$ can be considered as a function of $p$-valued logic and as a transformation of the ring $\Z/p\Z$. 

\begin {remark}
\label {rem-1}
If residues of the ring $\Z/p\Z$ are set as minimal non-negative integers, then operations in the ring $\Z/p\Z$ can be regarded as operations on the set $\{0, 1, \ldots, p-1\}.$
In this article, it will be convenient to use a special notation for such operations. Namely, we denote these operations  on the set $\{0, 1, \ldots, p-1\}$ by "$\oplus_p$" and "$\odot_p$" given as $x\oplus_py\equiv x+y \pmod p$ and $x\odot_py=x\cdot y \pmod p,$ correspondingly. 
\end {remark}


\subsection{$P$-adic dynamics}

Dynamical system theory studies trajectories (orbits), i.e.
sequences of  iterations of the function $f$:
$x_0, x_1=f(x_0), \ldots, x_{i+1}=f(x_i)=f^{(i+1)}(x_0), \ldots,$
where $f^{(s)}(x)=\underbrace{f(f(\ldots f(x))\ldots)}_{s}$.

We consider a $p$-adic autonomous dynamical system $\left\langle \Z_p,\mu_p ,f \right\rangle,$ for more details see, for example, \cite{Tfunc}-\cite{V1}, as well as \cite{Jeong}. 
The space $\Z_p$ is equipped with a natural probability measure $\mu_p$, namely,
the Haar measure ($\mu_p(B_{p^{-r}} (a))= p^{-r}$).

A measurable mapping $f\:\Z_p\>\Z_p$ is called measure-preserving if \\ $\mu_p(f^{-1}(U)) = \mu_p(U)$ for each measurable subset $U \subset \Z_p.$ 

Criteria of measure-preserving for 1-Lipschitz functions are presented in the following theorems.

\begin{theorem} (\cite {Unif0},  \cite {ANKH})
\label {cr_mer_Anashin}
A 1-Lipschitz functions $f:\Z_p \to \Z_p$ preserves the measure if and only if $f_{k-1}\equiv f\pmod{p^k}$ is bijective on $\Z/p^k\Z$ for any $k=1, 2, \ldots.$
\end{theorem}

\begin{theorem}[Theorem 2.1, \cite{MeraJNT}]
\label {cr_mer_vdp}
A 1-Lipschitz function $f:\Z_p\>\Z_p$ represented by the van der Put series (\ref{1-Lip})  preserves the measure if and only if 
\begin{enumerate}
\item $\{b_0, b_1, \ldots, b_{p-1}\}$ constitutes a complete set of residues modulo $p$ (i.e. $f(x)$ is bijective modulo $p$);
\item the elements in the set $\{b_{m+p^k}, b_{m+2p^k}, \ldots, b_{m+(p-1)p^k}\}$ are all nonzero residues modulo $p$ for any $m=0, \ldots, p^k-1,$ $k\ge 1.$
\end{enumerate}
\end{theorem}

\begin{theorem}[Theorem 3.1,  \cite{SOL}]
\label {cr_mer_coord}
A 1-Lipschitz function $f:\Z_p\>\Z_p$ represented in the coordinate form (\ref {coor3}) preserves the measure if and only if all functions $\varphi_0$ and $\varphi_{k,a_k},$ $a\in \{0, 1, \ldots, p^k-1\},$ $k\ge 1$ are bijective on $\{0, \ldots, p-1\}$.
\end{theorem}

\subsection {Automorphisms of algebraic systems}
\label {Automorph_def}

Recall that an algebraic system is an object $\langle \mathcal{A}, \Omega_{\mathcal{A}}, P_{\mathcal{A}} \rangle,$ where $\mathcal{A}$ is the carrier set, $\Omega_{\mathcal{A}}$ is the set of operations on $\mathcal{A}$ and $P_{\mathcal{A}}$ is the set of predicates on $\mathcal{A}.$ A predicate on the set $\mathcal{A}$ is considered as a characteristic function of the relation on this set (that is, the predicate determines the relation and vice versa). Further, we consider only binary operations and predicates (relations).

We remind that an automorphism of the algebraic system $\langle \mathcal{A}, \Omega_{\mathcal{A}}, P_{\mathcal{A}} \rangle$ is a bijective mapping $f: \mathcal{A} \to \mathcal{A}$ such that

\begin{enumerate}
    \item for any operation "$*$" from $\Omega_{\mathcal{A}}$ the map $f$ is a homomorphism with respect to the operation "$*$", that is $f(a*b)=f(a)*f(b)$ for $a, b \in \mathcal{A}$;
    \item for any predicate $\pi\in P_{\mathcal{A}},$ from $\pi(x,y)$ it follows $\pi(f(x),f(y))$ for $x,y \in \mathcal{A}$ (or in terms of relations, $x\rho_{\pi}y \Rightarrow f(x)\rho_{\pi}f(y)$, where $\rho_{\pi}$ is a relation defined by a predicate $\pi$).
\end{enumerate}

Hereinafter, we consider the algebraic system $\langle \mathcal{A}, \Omega_{\mathcal{A}}, P_{\mathcal{A}} \rangle,$ for which the carrier is a set of $p$-adic integers, namely, $\mathcal{A}=\Z_p.$ 

The family of predicates  $P_{\;\Z_p}$ determines the congruence relations  modulo $p^k,$ $k\ge 1.$    

A set of operations $\Omega_{\mathcal{A}}$ consists of one or two operations from the set $O_{\Z_p}=\{+, \cdot, \mathrm{XOR}, \mathrm{AND}\}$ given on $\Z_p.$ Here operations 
"$+$" and "$\cdot$" are arithmetical operations on $\Z_p$, and  coordinate-wise logical operations "$\mathrm{XOR}$" and "$\mathrm{AND}$" are defined in the following way. 
Let $p$-adic numbers $x, y \in \Z_p$ be defined in the canonical form.
Then, taking into account Remark \ref {rem-1}, we have 
\begin {align*}
x\mathrm{XOR}y&=(x_0\oplus_p y_0)+(x_1\oplus_p y_1)p+\ldots \\  
x\mathrm{AND}y&=(x_0\odot_p y_0)+(x_1\odot_p y_1)p+\ldots.
\end {align*}

In this paper, we consider algebraic systems of the form $\langle \Z_p, *, P_{\;\Z_p}\rangle$ or $\langle \Z_p, *_1, *_2, P_{\;\Z_p}\rangle$, where $*, *_1, *_2 \in O_{\Z_p}$. These algebraic systems differ only in the set of operations (the carrier and the set of predicates for these systems are fixed). Therefore, we shall specify only the operations under consideration to denote such algebraic systems. For example, through a $\mathcal{A}_p(*)$ we denote the algebraic system $\langle \Z_p, *, P_{\;\Z_p}\rangle$ for $*\in O_{\Z_p}.$

The set of all automorphisms of an algebraic system with respect to the operation of a composition of automorphisms forms a group which in our notation will be written in the form $Aut\mathcal{A}_p(*)$  (and $Aut\mathcal{A}_p(*_1, *_2)$ in the case of two operations). We denote an identity element  of the group of automorphisms  by $e.$


\section {Groups of automorphisms of $p$-adic integers}
\label {sec_automorph}
In this section we give a description of the groups of automorphisms of the following algebraic systems:
\begin{enumerate}
    \item $Aut\mathcal{A}_p(*),$ where $*\in O_{\Z_p},$ see subsection \ref {sec_1_operation};
    \item $Aut\mathcal{A}_p(*_1, *_2),$ where $*_1, *_2\in O_{\Z_p},$ see subsection \ref {sec_2_operation}.
\end{enumerate}

As shown in Theorem \ref{Prop_aut_2}, all groups of automorphisms  $Aut\mathcal{A}_p(*_1, *_2)$, where $*_1, *_2\in O_{\Z_p}$ are trivial groups (i.e., groups that have only one element). In this regard, in the section \ref{sec_3_new_operation} we consider the question of the existence of algebraic systems of the form $\langle \Z_p, g_1, g_2, \mathcal{P}\rangle,$ where $g_1, g_2$ are some "new" operations, for which the group of automorphisms  differs from the identity.

We use the apparatus developed in our previous works on $p$-adic dynamical systems, see, for example, \cite{MeraJNT} and \cite{SOL}, to describe the groups of automorphisms  of $p$-adic integers.

This possibility is explained by the following circumstances:
\begin{enumerate}
    \item a function $f: \Z_p\to \Z_p$ preserves all relations  modulo $p^k,$ $k\ge 1$ if and only if $f$ is a 1-Lipschitz function. Indeed, if $f(x)\equiv f(y) \pmod {p^k},$ $x, y\in \Z_p,$ $k\ge 1$ follows from $x\equiv y \pmod {p^k}$, then this is equivalent to $|f(x)-f(y)|_p\le |x-y|_p$;
    \item a composition of 1-Lipschitz functions is a 1-Lipschitz function. Indeed, $|f(g(x))-f(g(y))|_p\le |g(x)-g(y)|_p \le |x-y|_p$;
    \item a 1-Lipschitz function $f$ is bijective on $\Z_p$ if and only if $f$ preserves the measure, see  Corollary 4.5. from \cite{Subcoord_Hensel}; 
		\item a composition of measure-preserving 1-Lipschitz functions is a  measure-preserving 1-Lipschitz function.
\end{enumerate}

In terms of dynamical systems, the problem of describing automorphisms of $p$-adic integers reduces to describing the measure-preserving 1-Lipschitz functions, which preserve the operations of the considered algebraic system.

\subsection {Groups of automorphisms on $\Z_p$ with one operation}
\label {sec_1_operation}
In this section, we describe the groups of automorphisms of algebraic systems $Aut\mathcal{A}_p(*)$ for each operation from the set $O_{\Z_p}.$

Note that the functions that define the homomorphisms with respect to arithmetic operations "$+$" and "$\cdot$"  on the $p$-adic analogue of the field of complex numbers were considered in \cite{Schikhof}. In contrast to this case, we consider the functions that preserve the measure and define the homomorphism on $\Z_p$ for a wider set of binary operations. A full description of measure-preserving, 1-Lipschitz functions, which define homomorphisms for specific operations on $\Z_p,$ is presented in Theorem \ref{t_gom_ariff} (for arithmetic operations) and Theorem \ref{t_gom_coord} (for logical operations). 

\begin {theorem}[Arithmetic operations]
\label {t_gom_ariff}
$  $
\begin{enumerate}
    \item The group of automorphisms of the algebraic system $Aut\mathcal{A}_p(+)$ consists of functions $f:\Z_p\to\Z_p$ of the form:
$$
f(x)=Ax,
$$ 
where $A\in \Z_p$ and $A\not \equiv 0\pmod p.$
    \item The group of automorphisms of the algebraic system $Aut\mathcal{A}_p (\cdot)$ consists of functions  $f:\Z_p\to\Z_p$ of the form: 
\begin{equation}
\label {fun_umnz}
    f(x)=\begin{cases}
p^k A^k \theta^s(1+p\;t)^a, &\text{if} \;\; x=p^k\theta(1+t p),\\
0,                                               &\text{if}\;\; x=0
\end{cases}
\end{equation}
where $k\ge 0,$ $t, a, A \in \Z_p,$ $s\in \{1, \ldots, p-1\}$, $\theta\in \Z_p,\;\theta^{p-1}=1$ and
$$
A\not \equiv 0\pmod p, \;\;\;\; a\not \equiv 0\pmod p,\;\;\text{GCD}\;(s, p-1)=1.
$$ 
\end{enumerate}
\end {theorem}
     
\begin {proof} As we have already noted (see Section \ref{Automorph_def}) elements of $Aut\mathcal{A}_p(+)$ (or $Aut\mathcal{A}_p(\cdot),$ correspondingly) are 1-Lipschitz functions $f:\Z_p\to \Z_p$ (the condition of preserving the congruence relations modulo $p^k,$ $k\ge 1$). To prove this Theorem, we describe all 1-Lipschitz functions, which define a homomorphism with respect to the considered operation, and then, using the results for measure-preserving functions, we find the functions that are bijective on $\Z_p.$ As a result, we obtain a description of the elements of the groups $Aut\mathcal{A}_p(+)$ and $Aut\mathcal{A}_p(\cdot).$

Let $f$ defines a homomorphism with respect to the operation "$+$". Then $f(m)=m\cdot f(1),$ $m\in \Z.$ Let $f(1)=A\in \Z_p,$ $A\ne 0.$ Since 1-Lipschitz function $f$ is continuous on $\Z_p$ and $\Z$ is dense in $\Z_p$, then  $f(x)=A\cdot x,$ $x\in \Z_p$. The function $f(x)=A\cdot x$  preserves the measure if and only if $A\not\equiv 0\pmod p.$ It is clear that the function $f(x)=Ax$ defines a homomorphism with respect to addition on $\Z_p.$

Let $f$ defines a homomorphism with respect to multiplication on $\Z_p$. Let $f$ be distinct from the identity function. In particular, $f(0)=0.$ Indeed, if this is not so, then from  $f(0\cdot a)=f(0)f(a)=f(0)$ follows $f(a)=1$ for any $a\in \Z_p.$ In addition, we assume that there exists $a\in\Z_p$ such that $f(a)\ne 0$ (i.e. $f$ is a 
non-zero function).

We write each non-zero $p$-adic number with the aid of the Teichm\"{u}ller representation (see, for example, p. 81 in \cite {Schikhof}), namely in the following form:
\begin{equation}
\label {rep_p_adic}    
x=p^k\theta(1+pt), \; k\ge 0, \;t\in\Z_p, 
\end{equation}
and $\theta\in\Z_p,\;\;\theta^{p-1}-1=0$. 
Note that if $p=2,$ then $\theta=1$ and any non-zero $2$-adic number is represented as $x=2^k(1+2t).$

Let $p\ge 3$ and $T=\{1,\theta,\ldots,\theta^{p-2}\}$ be a set of all non-zero Teichm\"{u}ller representatives  (in other words, $T$ is the set of all solutions of equation $z^{p-1}-1=0$ in $\Z_p$). $T$ is a cyclic group (with respect to the operation of multiplication) generated, for example, by the element $\theta.$ Notice, that $f(T)=T.$ Indeed, let $f(\theta)=G\in\Z_p.$ Since $f$ is a homomorphism, then $G^{p-1}=1,$ i.e. $G\in T.$ If $G=\theta^s$ for some $s\in\{0, 1, \ldots, p-2\}$, then $f(\theta^r)=\theta^{rs}\in T.$ 
In particular, the homomorphism $f$ induces a mapping $f_T: T\to T$ of the form $z\mapsto z^s.$

As $f$ is a 1-Lipschitz function, then $1\equiv f(1)\equiv f(1+p\Z_p)\pmod p,$ i.e. $f(1+p\Z_p)\subset 1+p\Z_p.$ The set $1+p\Z_p$ forms a group with respect to the operation of multiplication. Indeed,  $(1+pt_1)(1+pt_2)=(1+p(t_1+t_2+pt_1t_2))\in 1+\Z_p$ for $t_1, t_2\in \Z_p$ and $1+p\Z_p$ is contained in the set of invertible elements of the ring $\Z_p.$ This means that $f$ induces a homomorphism  $\phi: 1+p\Z_p\to 1+p\Z_p.$ It is clear that $\phi$ is a 1-Lipschitz function (as a restriction of the 1-Lipschitz function $f$ to the set $1+p\Z_p$).  

Let $P=\{1, p, p^2, \ldots\}.$ Since $f$ is a homomorphism, then 
$$
f(P)=\{1, f(p), f(p^2), \ldots\}.
$$  
As $f$ is a 1-Lipschitz function, then $0\equiv f(0)\equiv f(p)\pmod p$, i.e. $f(p)=pA$ for some $A\in\Z_p.$

Thus, the function $f$, which defines a homomorphism with respect to multiplication on $\Z_p,$ can be represented in the form (taking into account the representation from the relation (\ref {rep_p_adic})):
\[
f(x)=\begin{cases}
p^k\cdot A^k\cdot \theta^s\cdot\phi(1+pt), &\text{if} \;\; x=p^k\theta(1+tp),\; k\ge 0,\\
0,                                                 &\text{if}\;\; x=0,
\end{cases}
\]
where $\theta\in \Z_p,\;\theta^{p-1}-1=0,$ $t\in \Z_p,$ $A \in \Z_p,$ $s\in \{0, 1, \ldots, p-2\}$ and a 1-Lipschitz function $\phi: 1+p\Z_p \to 1+p\Z_p$ defines a homomorphism with respect to multiplication on $1+p\Z_p.$

Let us find the representation of the function $\phi.$ Let $\mathrm{EXP}_p : p\Z_p\to 1+p\Z_p$ be the $p$-adic exponential function ($\mathrm{EXP}_2: 2^2\Z_2\to 1+2\Z_2$ for $p=2$) and $\mathrm {LN}_p: 1+p\Z_p\to p\Z_p$ be the $p$-adic logarithm ($\mathrm {LN}_2: 1+2\Z_2\to 2^2\Z_2$ for $p=2$). We consider the function $g: p\Z_p\to p\Z_p$ ($g: 2^2\Z_2\to 2^2\Z_2$ for $p=2$) such that $g(\tau)=\mathrm {LN}_p(\phi(\mathrm {EXP} _p(\tau))).$ Then, the function $g$ defines a homomorphism with respect to addition on $p\Z_p$ (on $2^2\Z_2$ for $p=2$):
\begin{align*}
g(\tau_1+\tau_2)=&\mathrm {LN}_p(\phi(\mathrm {EXP}_p(\tau_1+\tau_2)))=\\
&\mathrm {LN}_p(\phi(\mathrm {EXP}_p(\tau_1)\cdot\mathrm {EXP}_p(\tau_2)))=\\
&\mathrm {LN}_p(\phi(\mathrm {EXP}_p(\tau_1))\cdot\phi(\mathrm {EXP}_p(\tau_2)))=\\
&\mathrm {LN}_p(\phi(\mathrm {EXP}_p(\tau_1)))+\mathrm {LN}_p(\phi(\mathrm {EXP}_p(\tau_2)))=g(\tau_1)+g(\tau_2).
\end{align*}

Therefore, there exists $a\in \Z_p$ such that $g(\tau)=a\tau.$ Since $\mathrm {EXP}_p(\mathrm {LN}_p(1+pz))=1+pz,$ $z\in \Z_p,$ then 
$$
\mathrm {EXP}_p(g(\tau))=\mathrm {EXP}_p(a\cdot\tau)=\mathrm {EXP}_p(\tau)^a=\phi(\mathrm {EXP}_p(\tau)).
$$

Let $x=1+pt=\mathrm {EXP}_p(\tau),$ $\tau\in p\Z_p$ (and $\tau\in 2^2\Z_2$ for $p=2$). Then $\phi(x)=x^a,$ $a\in \Z_p.$ 

Thus, the function $f$ can be represented in the form 
$$
f(x)=f(p^k\theta(1+pt))=p^k\cdot A^k\cdot \theta^s\cdot(1+pt)^a.
$$

Performing the corresponding calculations, we see that the function of this type defines a homomorphism on $\Z_p$ with respect to multiplication.

Let us find the values $A, a\in \Z_p,$ $s\in \{1, 2, \ldots, p-1\},$ where the function $f$ of the form
(\ref{fun_umnz}) preserves the measure. For this, we use the criterion of Theorem \ref{cr_mer_vdp}. Let us find the value of the van der Put coefficients of the function $f.$ Let $t\in \{0, 1, \ldots, p^r-1\},$ $\theta \ne 0,$ $h\in \{1, 2, \ldots, p-1\},$ $k\ge 0.$ Then $B_0=f(0)=0$ and 
\begin{multline*}
b_{p^k\theta(1+p(t+p^{r}h)) \pmod {p^{k+r+1}}}=\frac {1}{p^{k+r}}B_{p^k\theta(1+p(t+p^{r}h)) \pmod {p^{k+r+1}}}\equiv \\
\equiv \frac {1}{p^{k+r}}\left(f(p^k\theta(1+p(t+p^{r}h)))-f(p^k\theta(1+p(t))\right)\equiv \\
\equiv a A^k\theta^sh\pmod p,\; r\ge 1,\;k\ge 1,
\end{multline*}
\begin{multline*}
b_{p^k\theta \pmod{p^{k+1}} }=\frac {1}{p^{k}}B_{p^k\theta \pmod{p^{k+1}}}\equiv \frac {1}{p^{k}} \left(f(p^k\theta)-f(0)\right)\equiv\\
\equiv A^k\theta^s \pmod p,\; r=0,\; k\ge 1,
\end{multline*}
\[
b_{\theta \pmod p}=B_{\theta \pmod p}\equiv f(\theta)\equiv \theta^s \pmod p,\; k=0.
\]

Since $\theta\not \equiv 0 \pmod p,$ then $\{b_{p^k\theta(1+p(t+p^{r}h)) \pmod {p^{k+r+1}}}: h=1, 2, \ldots, p-1\}$ coincides with the set of all non-zero residues modulo $p$ if and only if $a\not\equiv 0 \pmod p,$ and $A\not\equiv 0\pmod p.$ The set $\{b_{p^k\theta \pmod {p^{k+1}}}: \theta^{p-1}=1\},$ $k\ge 0$ coincides with the set of all non-zero residues modulo $p$ as $\text{GCD}\;(s, p-1)=1.$ Then, by Theorem \ref{cr_mer_vdp} the function $f$ preserves the measure if and only if 
$
a\not \equiv 0 \pmod p;\;\;\;A\not \equiv 0 \pmod p;\;\;\; \text{GCD}\;(s, p-1)=1.
$
\end {proof} 

\begin {remark}
If in (\ref{fun_umnz})  we set $a=n,$ $s=n,$ $A=p^{n-1}$ for some $n\in \N,$ then $f(x)=x^n.$ That is, all such polynomials define a homomorphism with respect to multiplication on $\Z_p$. Functions of the form $f(x)=x^n$ for $n > 1$ do not preserve the measure.   
\end {remark}

\begin {remark}
We note that each element (or function) $f\in Aut\mathcal{A}_p(\cdot)$ is uniquely determined by the set of parameters $(s,a,A),$ where $s\in \left(\Z/(p-1)\Z\right)^*$ and 
$a, A\in \Z_p^*.$ Here $\left(\Z/(p-1)\Z\right)^*$ is the group of units of the ring $\Z/(p-1)\Z$ and $\Z_p^*$ is the group of units of the ring $\Z_p$.
Let the elements (or functions) $f, g \in Aut_p\mathcal{A}(\cdot)$ be defined by the parameters $(s,a,A)$ and $(d,b,B).$ Then the composition $f(g)$ is determined by the parameters 
$$
\left(s\cdot d,\; a\cdot b,\; A\cdot (\theta_B)^s(1+pB_1)^a\right),
$$ where $B=\theta_B(1+pB_1).$
\end {remark}

\begin {theorem}[Logical operations]
\label {t_gom_coord}
$ $
\begin{enumerate}
    \item The group of automorphisms of the algebraic system $Aut\mathcal{A}_p(\mathrm{XOR})$ consists of  functions $f:\Z_p\to\Z_p$ given in the coordinate form:
\begin{equation*}
f(x)=f(x_0+\ldots+p^kx_k+\ldots)=\sum_{k=0}^{\infty}p^k\varphi_k(x_0,\ldots,x_k),
\end{equation*}
where $\varphi_k(x_0, \ldots, x_k)$ are $p$-valued logical functions and 
\[
\varphi_k(x_0,\ldots,x_k)=\alpha_{0}^{(k)}x_0\oplus_p\alpha_{1}^{(k)}x_1\oplus_p\ldots\oplus_p\alpha_{k}^{(k)}x_k,
\]
where $\alpha_{i}^{(k)}\in \{0, \ldots, p-1\},$ $0\le i\le k$ and $\alpha_{k}^{(k)}\not \equiv 0 \pmod p,$ $k\ge 0.$
    \item The group of automorphisms of the algebraic system $Aut\mathcal{A}_p(\mathrm{AND})$ consists of  functions $f:\Z_p\to\Z_p$ given in the coordinate form: 
\[
f(x)=f(x_0+px_1+\ldots+p^k x_k+\ldots)=\sum_{k=0}^{\infty}p^k(x_k^{s^{(k)}} \pmod p),
\]
where $\text{GCD}\;(s^{(k)}, p-1)=1,$ $k\ge 0.$
\end{enumerate}
\end {theorem} 
 
\begin {proof} 

According to Proposition 3.33 in \cite {ANKH}, a function represented in coordinate form 
\begin{equation*}
f(x)=f(x_0+\ldots+p^kx_k+\ldots)=\sum_{k=0}^{\infty}p^k\varphi_k(x_0,\ldots,x_k),
\end{equation*}
where $\varphi_k(x_0, \ldots, x_k)$ are $p$-valued logical functions, is a 1-Lipschitz function. 

Let $f$ defines a homomorphism with respect to the operation "$\mathrm{XOR}$" on $\Z_p$, i.e., $\varphi_k(x_0\oplus_py_0, \ldots, x_k\oplus_py_k)=\varphi_k(x_0, \ldots, x_k)\oplus_p\varphi_k(y_0, \ldots, y_k),$ $x_i, y_j\in\{0, 1, \ldots, p-1\},$ $k\ge 0.$
Let 
\[
\varphi_{k,r}(x)=\varphi_{k}(\underbrace {0,\ldots,0}_{r}, x,0,\ldots,0),\;0 \le  r\le k.
\]
Since $\varphi_{k,r}(x\oplus_py)=\varphi_{k,r}(x)\oplus_p\varphi_{k,r}(y),$ $x,y\in \Z/p\Z$, then $\varphi_{k,r}(x)$ is the homomorphism on $\Z/p\Z$ with respect to addition. Therefore, $\varphi_{k,r}(x)=a_r^{(k)}x,$ $a_r\in \Z/p\Z$ (i.e. $\Z/p\Z$ is a cyclic group with respect to the addition operation).
Since
\[
\varphi_k(x_0,\ldots,x_k)=\varphi_{k,0}(x_0)\oplus_p\ldots\oplus_p\varphi_{k,k}(x_k),
\]
then $\varphi_k(x_0, \ldots, x_k)=a_0^{(k)}x_0\oplus_pa_1^{(k)}x_1\oplus_p\ldots\oplus_pa_k^{(k)}x_k,$ $k\ge 0,$ where $a_i^{(j)}\in \{0, \ldots, p-1\}.$

It is clear that a function represented by the  coordinate functions  defines a homomorphism on $\Z_p$ with respect to the operation "$\mathrm{XOR}$".

Thus, the function
\[
f(x)=f(x_0+x_1p+\ldots)=\sum_{k=0}^{\infty} (a_0^{(k)}x_0\oplus_pa_1^{(k)}x_1\oplus_p\ldots\oplus_pa_k^{(k)}x_k)p^k
\]
defines a homomorphism on $\Z_p$ with respect to the operation $\mathrm{XOR}.$ 
Coordinate sub-functions of the function $f$ from the representation (\ref{coor3}) have the form $c\oplus_pa_k^{(k)}x_k,$ $c\in \{0, \ldots, p-1\}.$ These sub-functions are bijective on $\Z/p\Z$ as $a_k^{(k)}\not\equiv 0\pmod p.$ Thus, by Theorem \ref{cr_mer_coord} the function $f$ preserves the measure if and only if $a_k^{(k)}\not\equiv 0\pmod p,$ $k\ge 0.$

Let us prove the second statement of the theorem. Let $f$ be a homomorphism with respect to the operation "$\mathrm{AND}$" on $\Z_p,$ i.e., 
\begin{multline*}
\varphi_k(x_0\odot_p y_0,\ldots,x_k\odot_p y_k)=\\
=\varphi_k(x_0,\ldots,x_k)\odot_p \varphi_k(y_0,\ldots,y_k),\; x_i,y_j\in\{0,1,\ldots,p-1\},\; k\ge 0.
\end{multline*}

Let 
\[
\varphi_{k,r}(x)=\varphi_{k}(\underbrace {1,\ldots,1}_{r}, x,1,\ldots,1),\;0\le  r\le k.
\]
Since $\varphi_{k,r}(x\odot_p y)=\varphi_{k,r}(x)\odot_p\varphi_{k,r}(y),$ $x, y\in \Z/p\Z$, then $\varphi_{k,r}(x)$ is the homomorphism on $\Z/p\Z$ with respect to multiplication. Therefore, $\varphi_{k,r}(x)=x^{s_r^{(k)}}$ for $s_r^{(k)}\in \{0, 1, \ldots, p-1\}.$
Since
$$
\varphi_k(x_0, \ldots, x_k)=\varphi_{k,0}(x_0) \oplus_p\ldots \oplus_p\varphi_{k,k}(x_k),
$$
then $\varphi_k(x_0, \ldots, x_k)=a_0^{(k)}x_0\oplus_pa_1^{(k)}x_1\oplus_p\ldots\oplus_pa_k^{(k)}x_k,$ $k\ge 0,$ where $a_i^{(j)}\in \{0, \ldots, p-1\}.$

It is clear that a function, represented by the  coordinate functions,  defines a homomorphism on $\Z_p$ with respect to the operation  "$\mathrm{AND}$" and has the form:
\[
f(x)=f(x_0+x_1p+\ldots)=\sum_{k=0}^{\infty} (x_0^{s_0^{(k)}}\odot_p\ldots\odot_p x_k^{s_k^{(k)}})p^k.
\]

Coordinate sub-functions of the function $f$ from the representation (\ref{coor3}) have the form $a_0^{s_0^{(k)}}\odot_p\ldots \odot_pa_{k-1}^{s_{k-1}^{(k)}}\odot x_k^{s_k^{(k)}},$ $a_{i}\in \{0, \ldots, p-1\},$ $0\le i\le k-1.$ 
These sub-functions are bijective on $\Z/p\Z$ if and only if $s_0^{(k)}\equiv s_1^{(k)}\equiv\ldots \equiv  s_{k-1}^{(k)}\equiv 0 \pmod p$ and $\text{GCD}\;(s_k^{(k)}, p-1)=1.$ 

Thus, by Theorem \ref{cr_mer_coord} the function $f$ preserves the measure if and only if $s_0^{(k)}\equiv s_1^{(k)}\equiv\ldots \equiv s_{k-1}^{(k)}\equiv 0 \pmod p$ and 
$\text{GCD}\;(s_k^{(k)}, p-1)=1,$ $k\ge 0.$ To complete the proof, we put $s_k^{(k)}=s^{(k)},$ $k\ge 0.$
\end {proof}


\subsection {Groups of automorphisms on $\Z_p$ with two operations}
\label {sec_2_operation}
In this section we consider groups of automorphisms of algebraic systems of $p$-adic integers
$Aut\mathcal{A}_p(*_1, *_2)$, where $*_1, *_2\in O_{\Z_p}=\{+, \cdot, \mathrm{XOR}, \mathrm{AND}\}$ is the set of considered arithmetic and coordinate-wise logical operations. As we show in Theorem \ref{Prop_aut_2}, each of these  groups (a total of 6 groups of automorphisms) is trivial. We denote a trivial group by $e.$


\begin {theorem} 
\label {Prop_aut_2}
For the groups of automorphisms $Aut\mathcal{A}_p(*_1, *_2),$ where $*_1, *_2\in O_{\Z_p},$ the following relations hold:
\begin {align*} 
Aut\mathcal{A}_p(+,\cdot)=&Aut\mathcal{A}_p(+,\mathrm{XOR})=\\
&Aut\mathcal{A}_p(+, \mathrm{AND})=Aut\mathcal{A}_p(\;\cdot\;, \mathrm{XOR})=\\
&Aut\mathcal{A}_p(\;\cdot\;, \mathrm{AND})=Aut\mathcal{A}_p(\mathrm{XOR}, \mathrm{AND})=e.
\end {align*}
\end {theorem} 

\begin {proof}
Note that $Aut\mathcal{A}_p(*_1, *_2)=Aut\mathcal{A}_p(*_1)\cap Aut\mathcal{A}_p(*_2)$ for $*_1, *_2\in O_{\Z_p}.$
 
Let us show that $Aut\mathcal{A}_p(+,\cdot)=e.$  
Let $f(x)\in Aut\mathcal{A}_p(+)\cap Aut\mathcal{A}_p(\cdot).$ From Theorem \ref{t_gom_ariff} it follows that $f(x)=Ax,$ $A\in \Z_p.$  As $f(x)\in Aut\mathcal{A}_p(\cdot),$ then 
$f(1)=1.$ Thus, $f(1)=A=1,$ that is, $f(x)=x$ and $Aut\mathcal{A}_p(+,\cdot)=e.$ 

Let us show that $Aut\mathcal{A}_p(+,\mathrm{XOR})=e.$
 
Let $f(x)\in Aut\mathcal{A}_p(+)\cap Aut\mathcal{A}_p(\mathrm{XOR}).$ 
We see that $f=Ax,$ $A\in \Z_p,$ $A\not\equiv 0\pmod p$ by Theorem \ref{t_gom_ariff}. Let 
$$
A=a_0+a_1p+\ldots,\;a_k\in \{0, \ldots, p-1\},\;a_0\not\equiv 0\pmod p.
$$  
Then
\begin{equation}
\label {2_op_1}
f(x)=f(x_0+x_1p+\ldots)=\sum_{k=0}^{\infty}p^k\left(\sum_{s=0}^{k}a_{s}x_{k-s}\right).
\end{equation}

Let us show that $f(x)\equiv x\pmod {p^{k+1}},$ $k\ge 0.$ 
Set $k=0.$ It is clear that $f(x)\equiv a_0x_0\pmod p.$ A product of  $a_0x_0$ in $\Z_p$ can be represented as $a_0x_0=a_0x_0\pmod p+p\delta (x_0,a_0),$ where $\delta (a_0,x_0)$ 
is a $p$-valued function that reflects the transfer of the digit in operations on  $p$-adic numbers when written in canonical form. 

Notice, that $\frac{f(x)-f(x)\pmod p}{p}\equiv a_0x_1\oplus_p a_1x_0\oplus_p \delta(a_0,x_0)\pmod p.$ Since $f\in Aut\mathcal{A}_p(\cdot),$ then $\frac{f(x)-f(x)\pmod p}{p}$ is a linear $p$-valued function, see Theorem \ref {t_gom_coord}. Then $\delta(a_0,x_0)\pmod p$ is also linear function, i.e. $\delta(a_0,x_0)\equiv \alpha x_0\oplus_p \beta\pmod p.$ Set $x_0=0$ or $x_0=1,$ we obtain $0\equiv\delta(a_0,0)\equiv\beta\pmod p$ and $0\equiv\delta(a_0,1)\equiv\alpha\oplus_p\beta\equiv\alpha\pmod p$ (since digit transfer does not occur). In other words, $\delta(a_0,x_0)\equiv 0\pmod p$ for any $x_0\in \{0, \ldots, p-1\}.$ Then $a_0=1$ (in this case, the digit is not transferred for any value of $x_0$), so $f(x)\equiv x_0\pmod p.$ 

Let $f(x)\equiv x\pmod {p^k}.$ In particular, in the representation (\ref {2_op_1}) we have $a_1=\ldots=a_{k-1}=0$ and
\begin {multline*}
f(x)=f(x_0+x_1p+\ldots)=x_0+x_1p+\ldots+x_{k-1}p^{k-1}+\\
+(x_k+a_kx_{0})p^{k}+(x_{k+1}+a_kx_{1}+a_{k+1}x_{0})p^{k+1}+\ldots .
\end {multline*}
Let 
\begin{equation}
\label {2_op_2}    
x_k+a_kx_{0}=x_k\oplus_pa_kx_{0}+\delta(x_0,x_k)p,
\end{equation}
(here the function $\delta$ reflects the fact of digit transfer).
Thus 
$$
\frac{f(x)-f(x) \pmod {p^{k+1}}}{p^{k+1}}\equiv x_{k+1}\oplus_pa_kx_1\oplus_pa_{k+1}x_{0}\oplus_p\delta(x_0,x_k)\pmod p.
$$
Since $f\in Aut\mathcal{A}_p(\cdot),$ then $\frac{f(x)-f(x) \pmod {p^{k+1}}}{p^{k+1}}$ is a linear $p$-valued function, see Theorem \ref {t_gom_coord}. Therefore, 
$$
\delta(x_0,x_k)=\alpha_{k}x_k\oplus_p\alpha_{0}x_0\oplus_p\beta.
$$ 
Notice, that $\delta(x_0,x_k)\equiv 0$ with the following values of the variables $x_0=x_k=0;$ $x_0=1$ and $x_k=0;$ $x_0=0$ and $x_k=1,$ so, therefore, $\alpha_{k}=\alpha_{0}=\beta=0.$ Then $\delta(x_0,x_k)\equiv 0$ for any values of $x_0, x_k.$ Thus,  $a_k=0$ in (\ref {2_op_2}) (in this case, the digits are not transferred for any values of $x_0$ and 
$x_k$), i.e. $f(x)\equiv x\pmod {p^{k+1}}.$ As a result $f(x)\equiv x\pmod {p^{k+1}}$ for any $k\ge 0,$ i.e. $f(x)=x$ and $Aut\mathcal{A}_p(+,\mathrm{XOR})=e.$

Let us show that $Aut\mathcal{A}_p(+,\mathrm{AND})=e.$ 

Suppose that $f(x)\in Aut\mathcal{A}_p(+)\cap Aut\mathcal{A}_p(\mathrm{AND}),$ then we obtain $f(p^k)=p^k=Ap^k.$ Then $A=1$ and $Aut\mathcal{A}_p(+,\mathrm{AND})=e.$   

Let us show that $Aut\mathcal{A}_p(\cdot,\mathrm{XOR})=e.$ 
Let $f\in Aut\mathcal{A}_p(\cdot)\cap Aut\mathcal{A}_p(\mathrm{XOR}).$ As $1+p+p^2t=1\mathrm{XOR}(p+p^2t),$ then
\begin {multline}
\label {f_2_xor}
1+pA(1+pt)^a=1\mathrm{XOR}pA(1+pt)^a=1\mathrm{XOR}f(p(1+pt))=\\f(1\mathrm{XOR}(p+p^2t))=
f(1+p+p^2t)=(1+p+p^2t)^a.
\end {multline}
Set $t=0$ and differentiate functions from (\ref{f_2_xor}), then we get 
\begin{equation}
\label {f_3_xor} 
1+pA=(1+p)^a\;\; \text {and} \;\; A=\left(\frac {1}{1+pt}+p\right)^{a-1}.
\end{equation}

Set $t=0$ in (\ref {f_3_xor}), we get $A=(1+p)^{a-1}$ and $1+pA=(1+p)A.$ Thus $a=A=1.$ 

Using the representation of the second statement of the Theorem \ref{t_gom_ariff} and the second statement of Theorem \ref{t_gom_coord} for the function $f$ for $t\in \theta(1+p\Z_p),$ $\theta^{p-1}=1$ and setting $x_0\equiv\theta \pmod p,$ $x_0\in\{1, \ldots, p-1\}$ we obtain 
\begin {align*}
f(\theta(1+pt))\equiv&\theta^s(1+pt)\equiv \theta^s\equiv x_0^s \pmod p\\
f(\theta(1+pt))\equiv&\varphi_0(x_0)\equiv\alpha_{0}^{(0)}x_0\pmod p.
\end {align*}
Then $\alpha_{0}^{(0)}=1,$ $s=1.$ Thus, 
\[
f(x)=\begin{cases}
p^kt,&\text{if} \;\;x=p^kt,\; t\not \equiv 0 \pmod p,\;k\ge 0;\\
0,   &\text{if} \;\;x=0.
\end{cases}
\]
That is $f(x)=x$ and $Aut\mathcal{A}_p(\cdot,\mathrm{XOR})=e.$ 

Let us show that $Aut\mathcal{A}_p(\cdot,\mathrm{AND})=e.$ 
Let $f\in Aut\mathcal{A}_p(\cdot)\cap Aut\mathcal{A}_p(\mathrm{AND}).$
Set $x=p^kx_k,$ $x_k\in\{1, \ldots, p-1\}.$ Note that this number takes the form $x=p^k\theta(1+pt_{\theta})$ in  Teichm\"{u}ller representation, where $t_{\theta}\in\Z_p$ are choosen so that $\theta(1+pt_{\theta})\in\{1, \ldots, p-1\}$ and $\theta\equiv x_k\pmod p.$ Here we use the canonical representation of  $p$-adic numbers $\theta,\;\theta^{p-1}=1$ for the choice of such numbers $t_{\theta}.$

Taking into account Theorems \ref {t_gom_ariff} and \ref {t_gom_coord}, we obtain that
$$
p^kx_k^{s_k}\equiv f(x_kp^k)\equiv f(p^k\theta(1+pt_{\theta}))\equiv p^kA^k\theta^s\equiv p^kA^kx_k^s \pmod {p^{k+1}}, \; k\ge 0.
$$
Then, $s_k=s,$ $k\ge 0$ (moreover, $A\equiv 1\pmod p$).

Let $x=(1+pt),$ $t\in \{0, \ldots, p-1\}$ (we have $\theta=1,$ $k=0$ in the representation of $p$-adic numbers from item 2 in Theorem \ref{t_gom_ariff}). Taking into account the representation of $f$ from Theorem \ref{t_gom_ariff} and Theorem \ref{t_gom_coord}, we get
\begin {align*}
f(1+pt)\equiv &(1+pt)^a\equiv 1+atp\pmod {p^2};\\
f(1+pt)\equiv &1+(t^s \pmod p)p\pmod {p^2},
\end {align*}
i.e. $t^s\equiv at\pmod p.$ Since $s\in \{1, \ldots, p-1\}$, then $s=1$ and $s_k=s=1,$ $k\ge 0.$ Using the representation of $f$ from Theorem \ref{t_gom_coord}, we obtain 
$$
f(x_0+x_1p+\ldots)=\sum_{k=0}^{\infty}p^k(x_k^{s_k} \pmod p)=\sum_{k=0}^{\infty}p^kx_k. 
$$

That is $f(x)=x$ and $Aut\mathcal{A}_p(\cdot,\mathrm{AND})=e.$

Let us show that $Aut\mathcal{A}_p(\mathrm{XOR},\mathrm{AND})=e.$ 
Let $f\in Aut\mathcal{A}_p(\mathrm{XOR})\cap Aut\mathcal{A}_p(\mathrm{AND}).$ Using the coordinate representation of the function  $f$ (Theorem \ref {t_gom_coord}), we obtain
\[
\alpha_{0}^{(k)}x_0\oplus_p\ldots\oplus_p\alpha_{k-1}^{(k)}x_{k-1}\oplus_p\alpha_{k}^{(k)}x_k=x_k^{s_k^{(k)}},\;\; k\ge 0.
\] 
Then, $\alpha_{0}^{(k)}=\ldots=\alpha_{k-1}^{(k)}=0$ for $\alpha_{k}^{(k)}=1$ and $s_k^{(k)}=1.$
Thus, $f(x)=x$ and $Aut\mathcal{A}_p(\cdot,\mathrm{AND})=e.$
\end {proof}

\subsection {Groups of automorphisms on $\Z_p$ with "new" operations}
\label {sec_3_new_operation}

In connection with the results of section \ref {sec_2_operation}, we consider the question of the existence of algebraic systems of the form $Aut\mathcal{A}_p(g_1, g_2)=\langle \Z_p, g_1, g_2, \mathcal{P}_{\Z_p}\rangle$ for "new" pairs of binary operations $g_1, g_2$ on $\Z_p$  for which the group of automorphism  differs from the unit.
In particular, in Proposition \ref{new_operation} we describe all the "new" operations "$G$" (here we assume that these operations are given in the form of a series (\ref{op_new}) convergent on $\Z_p$), for which the group $Aut\mathcal{A}_p(+, G)$ is distinct from the identity group.

On the other hand, suppose that the "new" operations ($g_1, g_2$) are given using a formula in the basis of two arbitrary arithmetic and coordinate-wise logical operations over $\Z_p.$
In Proposition \ref{formuls}, we show that the necessary condition for the non-triviality of the group of automorphisms  $Aut\mathcal{A}_p(g_1, g_2)$ is that the set of formulas in the basis of the operations $g_1, g_2$ does not coincide with the set of formulas in basis of arithmetic operations over $\Z_p.$

Let the "new" operation "$G$" is given by a series converges on $\Z_p$ (it is sufficient to require that the general term of the series converges to zero in the $p$-adic metric), namely:

\begin{equation}
    \label {op_new}
G(x,y)=c+ax+by+\sum_{k=1}^{\infty}\;\sum_{i+j=n_k} c_{i,j}x^iy^j,\;\; c_{i,j}, a,b,c\in \Z_p,
\end{equation}
where for any $n_k\in\{n_1, n_2, \ldots \;|\; n_k\in \N,\; 1<n_1<n_2<\ldots\}=\mathcal{N}_G$ there exist $0 \le  i, j \le n_k$ such that $c_{i,j}\ne 0,$ and if $n\not \in \mathcal{N}_G,$ then $c_{i,j}=0$ for any $0 \le  i, j \le n,$ $i+j=n.$

\begin {proposition}
\label {new_operation}
Let the binary operation "$G$" on $\Z_p$ be defined by means of the series (\ref {op_new}). 

A group $Aut\mathcal{A}_p(+, G)\ne e$ if and only if the following relations hold for $\mathcal{N}_G\ne \emptyset:$
\begin{enumerate}
    \item $c=0;$
    \item $n_k=dq_k+1,$ $k\ge 1,$ where $d=p^s\cdot n,$ $q_k\in \N,$ $p\nmid n;$
    \item $\text {GCD} (n, p-1)\ne 1$ for $p\ne 2$ and $s=1$ for $p=2,$
\end{enumerate}
and relation $G=ax+by,$ $a, b\in \Z_p,$ $a, b\ne 0$ holds for $\mathcal{N}_G = \emptyset.$  

In other cases $Aut\mathcal{A}_p(+, G)=e.$ 
\end {proposition}

\begin {proof} 
Let $f\in Aut\mathcal{A}_p(+, G)=Aut\mathcal{A}_p(+)\cap Aut\mathcal{A}_p(G).$ From Theorem \ref {t_gom_ariff} it follows that $f=Ax,$ $A\ne 0.$ A function $f$ defines a homomorphism with respect to the operation "$G$", i.e., $AG(x,y)=G(Ax,Ay).$ Let $\mathcal{N}_G\ne \emptyset$. Using the representation (\ref{op_new}), we get that  $A$ satisfies the system of equations 
$$
A^{n_1}=A,\;A^{n_2}=A,\;\ldots,\; A^{n_k}=A,\; \ldots
$$
or $A^{n_k-1}=1,$ $k\ge 1.$ This system of equations is equivalent to the equation $A^d=1.$ Let $x=0,$ $y=0,$ then $Ac=c$ and $c=0.$ It is easy to see that under the conditions of the proposition, the function $f(x)=Ax,$ $A^d=1$  defines a homomorphism with respect to the operation "$G$". Then the condition $Aut\mathcal{A}_p(+, G)\ne e$ is equivalent to the fact that the equation $A^d=1$ in $\Z_p$ has more than one solution.

Let $d=p^s\cdot n,$ $p\nmid n.$ If $p\nmid d,$ then the equation $A^d=1$ has $\text {GCD} (d,p-1)$ solutions in $\Z_p$ (see, for example, Theorem 3.24 in \cite {ANKH}). If $d=p^s,$ then the equation $A^{p^k}=1$ has a unique solution $A=1,$ except when $p=2$ and $k=1$ (in this case, the equation $A^2=1$ in $\Z_2$ has solutions $A=1,$ $A=-1;$ see, for example, Theorem 3.36 in \cite {Katok}). Clearly, for $d=p^sn,$ $p\nmid d$ the equation $A^d=1$ has $\text {GCD} (n,p-1)$ solutions in $\Z_p.$ 
Thus, if $p\ne 2,$ then the equation $A^d=1$ has exactly $\text {GCD} (n,p-1)$ solutions in $\Z_p.$ If $p=2,$ then for $s=0$ and $s\ge 2$ the equation $A^d=1$ has a unique solution $A=1$ in $\Z_2.$
If $s=1$ (that is, $d=2n$), then the equation $A^d=1$ has two solutions $A=\pm 1$ in $\Z_2.$

If $\mathcal{N}_G=\emptyset$, then $G(x,y)=c+ax+ay$ and $c=0.$  The function $f(x)=Ax$ defines the automorphism on $\mathcal{A}_p(+, G)$ for any $A\ne 0.$ Since $G$ is a binary operation by the initial condition, then $a, b\ne 0.$ 
\end {proof}

\begin {remark}
From Proposition \ref {new_operation} it follows that if the  group $Aut\mathcal{A}_p(+, G)\ne e$ ($G$ is a binary operation on $\Z_p$), then $Aut\mathcal{A}_p(+, G)$ is either finite and consists of $r\ne 1$ elements, where $r$ is a divisor of $p-1;$ or is infinite and $Aut\mathcal{A}_p(+,G)=Aut\mathcal{A}_p(+)\cong \Z_p^*$ (here $G(x,y)=ax+by,$ $a, b\ne 0$).
\end {remark}

\begin{example}
\label {example_operations}
Let us   present some examples of operations "$G$" for $p\ne 2,$ for which the $Aut\mathcal{A}_p(+, G)\ne e.$ 
$$
G(x,y)=xy^{p-1},\;G(x,y)=x^{p-1}y+xy^{p-1},\;G(x,y)=x^{\frac{p-1}{2}}\cdot y^{\frac{p+1}{2}}.
$$
For all these operations, the groups $Aut\mathcal{A}_p(+, G)$ consist of a $p-1$ elements.
\end{example}

Now let us  consider the case when the binary operations $g_1, g_2: \Z_p\times \Z_p \to \Z_p$ for the algebraic system $Aut\mathcal{A}_p(g_1, g_2)$ are given by formulas in the basis of operations $*_1, *_2\in O_{\Z_p}.$ That is $g_1, g_2$ are expressed through a pair of arithmetic or coordinate-wise logical operations.

By analogy with formulas of Boolean algebra, let us define formulas with respect to the operations $g_1$ and $g_2$ over $\Z_p:$
\begin{enumerate}
    \item  elements of $\Z_p,$ variables and operations $g_1, g_2$ are formulas;
    \item if $F_1$ and $F_2$ are formulas, then $g_1(F_1, F_2)$ and $g_2(F_1, F_2)$ are formulas.    
\end{enumerate}

We denote the set of all formulas defined with respect to operations $g_1$ and $g_2$ as $[g_1, g_2].$  The following assertion holds.

\begin {proposition}
\label {formuls}

Let operations $g_1,\;g_2:\Z_p\times \Z_p \to \Z_p$ be defined by the formulas from $[*_1, *_2],$ $*_1, *_2\in O_{\Z_p}$ (arithmetic and coordinate-wise logical operations) and 
$Aut \mathcal{A}_p(g_1, g_2)\ne e.$ 

Then $[g_1, g_2]\subset [*_1, *_2]$ and $[g_1, g_2]\ne [*_1, *_2].$ 
\end {proposition}

\begin {proof}
Since $g_1, g_2 \in [*_1, *_2],$ then $[g_1, g_2]\subset [*_1, *_2].$ Assume that $[g_1, g_2]=[*_1, *_2],$ $*_1, *_2 \in O_{\Z_p}=\{+,\;\cdot, \mathrm{XOR}, \mathrm{AND}\}.$ Then "$*_1$" and "$*_2$" are defined by the formulas $\Psi_1(x_1,x_2)$ and $\Psi_2(x_1,x_2)$ with respect to the operations $g_1$ and $g_2.$ Let $f\in Aut\mathcal{A}_p(g_1, g_2).$ Since $f$ is a homomorphism with respect to $g_1$ and $g_2,$ then
$$
f(a*_ib)=f(\Psi_i(a,b))=\Psi_i(f(a),f(b))=f(a)*_if(b),\;i=1, 2,
$$
i.e. $f$ is the homomorphism with respect to "$*_1$" and "$*_2$". Then from Theorem \ref {Prop_aut_2}
it follows that $f$ is an identity mapping, i.e. $Aut\mathcal{A}_p(g_1, g_2)=e.$ This contradicts with the condition of the Proposition.
\end {proof}


\section{Automorphisms of $p$-adic integers and fully homomorphic ciphers}
\label {sec_model}

As we have already noted, the motivation for studying groups of automorphisms  of  $p$-adic integers is the problem of the existence of fully homomorphic ciphers. The connection between these problems is explained by the fact that for a wide family of ciphers, one can construct their "continuous" $p$-adic model. 
In this model, the ciphers are described by the algebraic system $\langle \Z_p, *_1, *_2, P_{\Z_p} \rangle,$ for which the family of predicates $P_{\Z_p}$ determines congruence relations modulo $p^k,$ $k\ge 1,$  operations "$*_1$", "$*_2$" will be selected from the set $O_{\Z_p}=\{+, \cdot,\; \mathrm{XOR}, \mathrm{AND}\}.$

In this section, we describe such family of ciphers and their "continuous" $p$-adic model.

The general idea of a fully homomorphic encryption is as follows (see, for example, \cite {Obz_appl_2} and \cite {Obz_appl_1}). Suppose we have a  set of data $M$. The operations 
$g_1: M\times M \to M,$ $g_2: M\times M \to M$ are defined on the set $M$. It is necessary to find the value of an expression $W(d_1, \ldots, d_n)$, which is defined through the operations $g_1$ and $g_2$ on the data  $d_1, \ldots, d_n\in M.$ 

By analogy with the formulas of Boolean algebra, the expression $W$ can be considered as a formula on the basis $g_1$ and $g_2.$ If the calculation of the formula $W$ is performed remotely (for example, using cloud services), then the user sends the data $d_1, \ldots, d_n$ to an untrusted environment (for example, to the cloud server). After that, the calculation result returns to the user. In this case, the user's data become open. 

We understand a cipher as a family of bijective transformations $f_a$ of the set $M,$ where each transformation is identified by a certain parameter $a$ -- the encryption key. Suppose that $f_a$ is a homomorphism with respect to the operations $g_1$ and $g_2.$ Then, $f_a(W(d_1, \ldots, d_n))=W(f_a(d_1), \ldots, f_a(d_n)).$ This means that the remote computations are performed on encrypted data $f_a(d_1), \ldots, f_a(d_n)$ and the result of calculations $W(d_1, \ldots, d_n)$ is obtained in encrypted form $f_a(W).$ That is, only the user has access to the data $d_1, \ldots, d_n.$ In general, this approach provides complete trust in remote computing.

Next, we give a description of the family of ciphers, for which we will consider their "continuous" $p$-adic model.

Let us remind that a cipher is a set $\langle X, R, Y, h_r, r\in R\rangle,$ where $X$ is a set of plain texts, $Y$ is a set of cipher-texts, $R$ is a set of keys, encryption functions $h_r$ are defined by the parameter $r\in R$ and define an injective  map $X\to Y.$ Here we assume that all maps $h_r$ are surjective.

A family of ciphers $\mathcal{C}_p=\langle X, R, Y, h_r, r\in R\rangle$ we set in the following way:
\begin{enumerate}[a)]
\item Let $X=Y=X^{(\infty)}$ be a set of all words (as a sequence of finite length) in the alphabet $\mathcal{B}=\{0, 1, \ldots, p-1\}$ for prime number $p$ (if we denote as $X^{(k)}$ a set of all words of the length $k$ in the alphabet $\mathcal{B},$ $k\ge 1,$ then $X^{(\infty)}=\cup_{k=1}^{\infty}X^{(k)}$).
\item Functions $h_r : X^{(\infty)}\to X^{(\infty)},$ $r\in R$ satisfy the following conditions: 
\begin{multline}
1.\; h_r : X^{(k)}\to X^{(k)},\;h_r \;\; \text {are bijective on}\;\; X^{(k)} for k\ge 1.
\end{multline}
\begin{multline}
\label {line 3.2}
2.\;\;\; \text {if}\;\; \{x_1,\ldots,x_s,\ldots, x_{k}\}\stackrel{h_r}{\longmapsto} \{y_1,\ldots,y_s, \ldots, y_{k}\} \;\;\text{then}\\
 \{x_1, \ldots, x_s\} \stackrel{h_r}{\longmapsto} \{y_1,\ldots,y_s\} \;\; \text {for any} \;\; 1\le s\le k.
\end{multline}
\end{enumerate}
For ciphers from the family $\mathcal{C}_p,$ we define operations on the set $X^{(\infty)}.$  Let $x, y\in X^{(k)},$  and $\tau_k: X^{(k)}\to \{0, 1, \ldots, p^k-1\},$ $k\ge 1,$
$$
\tau_k(x)=\tau_k(\{x_1, \ldots, x_{k}\})=x_1+x_2p+\ldots+x_{k}p^{k-1}.
$$
The following operations are defined on the set $X^{(k)}:$
\begin{align*}
x+_ky=&\tau_k^{-1}\left(\tau_k(x)+\tau_k(y) \pmod {p^k}\right);\\ x\cdot_ky=&\tau_k^{-1}\left(\tau_k(x)\cdot\tau_k(y) \pmod {p^k}\right);\\
x\mathrm{XOR}_ky=&\tau_k^{-1}\left(\tau_k(x)\mathrm{XOR}\tau_k(y) \pmod {p^k}\right);\\
x\mathrm{AND}_ky=&\tau_k^{-1}\left(\tau_k(x)\mathrm{AND}\tau_k(y) \pmod {p^k}\right).
\end{align*}
The set of such operations, we denote as $\overline {O}_p.$

Note that, the family $\mathcal{C}_p$ contains substitution ciphers, substitution ciphers streaming, keystream ciphers (in the alphabet of $p$ elements). On the other hand, there are no ciphers in $\mathcal{C}_p$ with different parameters of the sets of plain-text and cipher-text (for example, when the number of elements in the alphabet is a composite integer).   

As a "continuous" $p$-adic model of ciphers from the family $\mathcal{C}_p$ with given operations from the set $\overline {O}_p,$ we consider the algebraic system $\langle \mathcal{A}, \Omega_{\mathcal{A}}, P_{\mathcal{A}} \rangle,$ where:
\begin{enumerate}
    \item an algebraic system carrier is $\mathcal{A}=\Z_p;$
    \item a family of predicates $P_{\mathcal{A}}$ determines the congruence relation modulo $p^k,$ $k\ge 1;$
    \item as operations from $\Omega_{\mathcal{A}},$ we consider any pair of operations from the set $O_{\Z_p}$ (arithmetic and coordinate-wise logical operations on $\Z_p$);
    \item the automorphisms of the algebraic system $\langle \mathcal{A}, \Omega_{\mathcal{A}}, P_{\mathcal{A}} \rangle$ correspond to the transformations of the open and ciphered texts ($h_r$) for ciphers from $\mathcal{C}_p.$
\end{enumerate}

Taking into account the previously used notation, such algebraic systems we denote $\mathcal{A}_p(*_1, *_2),$ $*_1, *_2 \in O_{\Z_p}.$ The choice of such a model is determined by the following circumstances:

1. Let $x=\{x_1, \ldots, x_{k}\}\in X^{(k)},$ $k\ge 1$ 
(this is an element from the set of plain- and cipher-texts for the ciphers from $\mathcal{C}_p$). An element $x$ we associate with the element $\tau_k(x)\in \Z/p^{k}\Z.$ Then the set $\cup_{k\ge 1}X^{(k)}$ we, naturally, associate with the projective limit of residue rings $\Z/p^{k}\Z$ with respect to the natural projections $\Z/p^{k+1}\Z \to \Z/p^{k}\Z.$ Since $\varprojlim\Z/p^{k}\Z=\Z_p,$ then the set $\cup_{k\ge 1}X^{(k)}$ has been associated with the ring of  $p$-adic integers $\Z_p.$

2. Let $x=\{x_1, \ldots, x_{k}\}\stackrel{h_r}{\longmapsto} \{y_1, \ldots, y_k\}=y$ and $f_r^{(k)}: \Z/p^k\Z \to \Z/p^k\Z,$ $f_r^{(k)}(\tau_k(x))=\tau_k(y),$ $k\ge 1.$

Taking into account the condition (\ref {line 3.2}) for $h_r$, we obtain that $f_r^{(k)}$ define a 1-Lipschitz function $f_r:\Z_p \to \Z_p$ such that $f_r\equiv f_r^{(k)} \pmod {p^k},$ $k\ge 1$ (in particular, $f_r$ retains all congruence relations modulo $p^k,$ $k\ge 1$).

From the bijectivity of $h_r$ on $X^{(k)},$ $k\ge 1$ (the third property for $h_r$) and the method of determining $f_r,$ it follows that  $f_r^{(k)}$ (considering (\ref {_x})) are bijective on $\Z/p^k\Z$ for $k\ge 1.$ By Theorem \ref {cr_mer_Anashin} the functions $f_r$ preserve the measure. As we have already noted, the property of a measure-preservation means that $f_r$ is bijective on $\Z_p.$

3. It is clear, that operations from $\overline {O}_p$ can be extended by continuity on $\Z_p,$ and these extensions correspond to the arithmetic and coordinate-wise logical operations on $\Z_p$ (i.e. operations from $\overline {O}_p$ can be extended to $O_{\Z_p}$ by the continuity). 

If encoding transformations of ciphers from $\mathcal{C}_p$ define fully homomorphic ciphers with respect to any pair of operations $\overline {*}_1, \overline {*}_2\in \overline {O}_p,$ then these transformations correspond automorphisms  $f_r$ of the algebraic system $\mathcal{A}_p(*_1, *_2),$ here $*_1,\; *_2$ are operations on $\Z_p,$ which correspond to operations $\overline {*}_1, \overline {*}_2,$ given on sets of plain- and cipher-texts for ciphers from $\mathcal{C}_p.$

Let us give examples of a representation of ciphers from  $\mathcal{C}_p$ within our model.

\begin{example}
The symmetric permutation group on $\mathcal{B}=\{0, 1, \ldots, p-1\}$ we denote by $S_p$ ($\mathcal{B}$ is the alphabet of plain- and cipher-texts of ciphers from $\mathcal{C}_p$).
Let $x=\{x_1, \ldots, x_k, \ldots\}\in X^{(\infty)}.$ The action of permuting $g\in S_p$ on an element $\alpha\in \mathcal{B}$ we denoted by $\alpha^g.$
In a $p$-adic model, the encryption function $h_r$ is modelled by a 1-Lipschitz function:  
\begin{align*}
 f_r(x)=&\sum_{k=0}^{\infty}p^kx_k^{g_k} \;\;\; \text {for substitution ciphers streaming};\\
 f_r(z)=&\sum_{k=0}^{\infty}p^k(x_k\oplus_p\gamma_k) \;\;\;
 \text {for keystream ciphers} ;\\
 f_r(z)=&\sum_{k=0}^{\infty}p^kx_k^{g} \;\;\;\;\;\;\;\;\;\;\;\;\;\;\; \text {for substitution ciphers}.
 \end{align*}

\end{example}

In conclusion, we note that the results of Theorem \ref {Prop_aut_2} show that  there are no fully homomorphic ciphers with respect to each pair of operations from $\overline {O}_p$ in the family of ciphers $\mathcal{C}_p.$
On the other hand, Propositions \ref {new_operation} and \ref {formuls} show that there is a potential possibility for the existence of fully homomorphic ciphers with respect to "new" operations. However, in this case, the possibilities of computations in the basis of "new" operations are limited (in comparison with calculations on the basis of arithmetic and coordinate-wise logical operations).



\begin{thebibliography}{99}



\bibitem{Tfunc} V. Anashin, A. Khrennikov, E. Yurova, T-functions revisited: new criteria for bijectivity/transitivity,  {\it Designs, Codes and Cryptography}, Springer US, (2012) 1-25.

\bibitem{Unif0} V. Anashin, Uniformly distributed sequences of $p$-adic integers, II, {\it Discrete Math. Appl.}, 12(6)(2002) 527--590.

\bibitem{Erg} V. Anashin, Ergodic Transformations in the Space of $p$-adic Integers, in: $p$-adic
Mathematical Physics. 2-nd Int. Conference (Belgrade, Serbia and Montenegro,
21 September 2005), {\it AIP Conference Proceedings}, 826(2006) 3--24.

\bibitem{An_avt_1} V. Anashin, Automata finiteness criterion in terms of van der Put series of automata functions,
{\it P-Adic Numbers, Ultrametric Analysis, and Applications,} 4(2) (2012)  151--160. 

\bibitem{ANKH} V. Anashin, A. Khrennikov, {\it Applied Algebraic Dynamics}, de Gruyter
Expositions in Mathematics vol~ 49, Walter de Gruyter (Berlin --- New York), 2009.  

\bibitem{V00} D. K. Arrowsmith, F. Vivaldi, Geometry of $p$-adic Siegel discs, {\it Physica D} 71 (1994) 222--236.

\bibitem{V1} D. Bosio and F. Vivaldi, Round-off errors and $p$-adic numbers, {\it Nonlinearity} 13 (2000) 309--322.

\bibitem{Coh} P. M. Cohn, Universal Algebra, D. Reidel Publishing Company.

\bibitem{Obz_appl_2} C. Fontaine, F. Galand, A survey of homomorphic encryption for nonspecialists,
 {\it EURASIP Journal on Information Security} 1 (2007) 41--50.

\bibitem{Jeong} S. Jeong, Toward the ergodicity of $p$-adic 1-Lipschitz functions represented by the van der Put series, {\it Journal of Number Theory} vol~ 133, Issue 9 (2013) 2874--2891.

\bibitem{Katok} S. Katok, {\it $p$-adic Analysis Compared with Real}, Student Mathematical Library, AMS, vol~ 37, 2007.

\bibitem{MeraJNT} A. Khrennikov, E. Yurova,  Criteria of measure-preserving for $p$-adic dynamical systems in terms of the van der Put basis. 
Journal of Number Theory, 133(2) (2013) 484-491.

\bibitem{SOL} A. Khrennikov, E. Yurova, Criteria of ergodicity for $p$-adic dynamical systems in terms of coordinate functions. {\it Chaos, Solitons \& Fractals} vol~ 60 (2014) 11-30. 

\bibitem{KHR97} A. Khrennikov, {\it Non-Archimedean analysis: quantum
paradoxes, dynamical systems and biological models}, Kluwer,
Dordreht, 1997.

\bibitem{Subcoord_Hensel} A. Yu. Khrennikov, E. I. Yurova Axelsson, Subcoordinate Representation of $p$-adic Functions and Generalization of Hensel Lemma, To be published in {\it Izvestiya: Mathematics} 82 (2018).

\bibitem{Maltcev} A. I. Mal'cev, Algebraic Systems, Springer-Verlag, 1973.

\bibitem{Obzor_Ind} P. V Parmar, S. B Padhar, S. N Patel, N. I Bhatt, R. H Jhaveri, Survey of various homomorphic encryption algorithms and schemes, {\it International
Journal of Computer Applications} 91(8) (2014)  26--32.

\bibitem{V0} J. Pettigrew, J. A. G. Roberts and F. Vivaldi, {\it Complexity of regular invertible $p$-adic motions, Chaos} 11 (2001) 849--857.

\bibitem{Obz_appl_1} D. K. Rappe, {\it Homomorphic cryptosystems and their applications}, 2006.

\bibitem{Schikhof} W.H. Schikhof, {\it Ultrametric calculus. An introduction to $p$-adic analysis}, Cambridge: Cambridge University Press,  1984.

\bibitem{YuRecent} E. Yurova Axelsson, On recent results of ergodic property for $p$-adic dynamical systems, {\it P-Adic Numbers, Ultrametric Analysis, and Applications} vol~ 6, Issue 3 (2014) 235-257.

\bibitem{Subcoord_Hensel_1} E. Yurova Axelsson, A. Khrennikov, Generalization of Hensel lemma:
finding of roots of $p$-adic Lipschitz functions, {\it Journal of Number Theory} 158 (2016) 217--233. 

\end{thebibliography}
\end{document}